\documentclass{amsart}

\usepackage{amssymb}

\usepackage{tikz}

\theoremstyle{plain}
\newtheorem{theorem}{Theorem}
\newtheorem{lemma}[theorem]{Lemma}
\newtheorem{proposition}[theorem]{Proposition}

\theoremstyle{definition}
\newtheorem{definition}[theorem]{Definition}

\theoremstyle{remark}
\newtheorem{remark}[theorem]{Remark}
\newtheorem*{acknowledgments}{Acknowledgments}

\tikzset{myarrow/.style={->,>=stealth,line width=0.5pt}}

\DeclareMathOperator{\Gal}{Gal}

\begin{document}

\title{Rational maps with rational multipliers}

\author{Valentin Huguin}
\address{Jacobs University Bremen gGmbH, Campus Ring 1, 28759 Bremen, Germany}
\email{v.huguin@jacobs-university.de}
\thanks{The research of the author was supported by the German Research Foundation (DFG, project number 455038303)}


\begin{abstract}
In this article, we show that every rational map whose multipliers all lie in a given number field is a power map, a Chebyshev map or a Latt\`{e}s map. This strengthens a conjecture by Milnor concerning rational maps with integer multipliers, which was recently proved by Ji and Xie.
\end{abstract}

\maketitle

\section{Introduction}

Suppose that $f \colon \widehat{\mathbb{C}} \rightarrow \widehat{\mathbb{C}}$ is a rational map. A point $z_{0} \in \widehat{\mathbb{C}}$ is said to be \emph{periodic} for $f$ if there exists an integer $p \geq 1$ such that $f^{\circ p}\left( z_{0} \right) = z_{0}$. In this case, the least such integer $p$ is called the \emph{period} of $z_{0}$ and $\left\lbrace f^{\circ n}\left( z_{0} \right) : n \geq 0 \right\rbrace$ is said to be a \emph{cycle} for $f$. The \emph{multiplier} of $f$ at $z_{0}$ is the unique eigenvalue $\lambda \in \mathbb{C}$ of the differential of $f^{\circ p}$ at $z_{0}$, so that $\lambda = \left( f^{\circ p} \right)^{\prime}\left( z_{0} \right)$ if $z_{0} \in \mathbb{C}$. The map $f$ has the same multiplier at each point of the cycle. Furthermore, the multiplier is invariant under conjugation: if $\phi \colon \widehat{\mathbb{C}} \rightarrow \widehat{\mathbb{C}}$ is a M\"{o}bius transformation and $g = \phi \circ f \circ \phi^{-1}$, then $\phi\left( z_{0} \right)$ is periodic for $g$ with period $p$ and multiplier $\lambda$.

The notion of multiplier is fundamental in complex dynamics and plays a major role in the study of both local and global dynamics of rational maps. Furthermore, multipliers almost determine rational maps up to conjugation: aside from flexible Latt\`{e}s maps, there are only finitely many conjugacy classes of rational maps that have the same collection of multipliers at each period (see~\cite[Corollary~2.3]{MM1987}).

The purpose of this article is to show that the multipliers of a rational map do not all lie in a given number field unless the rational map is exceptional.

\begin{definition}
A rational map $f \colon \widehat{\mathbb{C}} \rightarrow \widehat{\mathbb{C}}$ of degree $d \geq 2$ is said to be a \emph{power map} if it is conjugate to $z \mapsto z^{\pm d}$.
\end{definition}

For every $d \geq 2$, there exists a unique polynomial $T_{d} \in \mathbb{C}[z]$ such that \[ T_{d}\left( z +z^{-1} \right) = z^{d} +z^{-d} \, \text{.} \] The polynomial $T_{d}$ is monic of degree $d$ and is called the $d$th \emph{Chebyshev polynomial}.

\begin{definition}
A rational map $f \colon \widehat{\mathbb{C}} \rightarrow \widehat{\mathbb{C}}$ of degree $d \geq 2$ is said to be a \emph{Chebyshev map} if it is conjugate to $\pm T_{d}$.
\end{definition}

Power maps and Chebyshev maps are exactly the finite quotients of affine maps on cylinders (see~\cite[Lemma~3.8]{M2006}). There are also finite quotients of affine maps on tori.

\begin{definition}
A rational map $f \colon \widehat{\mathbb{C}} \rightarrow \widehat{\mathbb{C}}$ of degree $d \geq 2$ is said to be a \emph{Latt\`{e}s map} if there exist a torus $\mathbb{T} = \mathbb{C}/\Lambda$, with $\Lambda$ a lattice in $\mathbb{C}$, a holomorphic map $L \colon \mathbb{T} \rightarrow \mathbb{T}$ and a nonconstant holomorphic map $p \colon \mathbb{T} \rightarrow \widehat{\mathbb{C}}$ that make the following diagram commute:
\begin{center}
\begin{tikzpicture}
\node (M00) at (0,0) {$\mathbb{T}$};
\node (M01) at (1.5,0) {$\mathbb{T}$};
\node (M10) at (0,-1.5) {$\widehat{\mathbb{C}}$};
\node (M11) at (1.5,-1.5) {$\widehat{\mathbb{C}}$};
\draw[myarrow] (M00) to node[above]{$L$} (M01);
\draw[myarrow] (M10) to node[below]{$f$} (M11);
\draw[myarrow] (M00) to node[left]{$p$} (M10);
\draw[myarrow] (M01) to node[right]{$p$} (M11);
\end{tikzpicture}
\end{center}
Furthermore, if $p$ has degree $2$ and $L$ is of the form \[ L \colon z +\Lambda \mapsto a z +b +\Lambda \, \text{,} \quad \text{with} \quad a \in \mathbb{Z}, \, b \in \mathbb{C} \, \text{,} \] then $f$ is said to be a \emph{flexible} Latt\`{e}s map.
\end{definition}

Power maps, Chebyshev maps and Latt\`{e}s maps play a special role in complex dynamics and are sometimes called exceptional. For example, their multipliers all lie in a discrete subring of $\mathbb{C}$.

\begin{proposition}[{\cite[Corollary~3.9]{M2006}}]
\label{proposition:specialPC}
Suppose that $f \colon \widehat{\mathbb{C}} \rightarrow \widehat{\mathbb{C}}$ is a power map or a Chebyshev map. Then $f$ has only integer multipliers.
\end{proposition}

An \emph{imaginary quadratic field} is a number field of the form $\mathbb{Q}\left( i \sqrt{D} \right)$, with $D$ a positive integer. Given a number field $K$, we denote by $\mathcal{O}_{K}$ its ring of integers.

\begin{proposition}[{\cite[Corollary~3.9 and Lemma~5.6]{M2006}}]
\label{proposition:specialL}
Suppose that $f \colon \widehat{\mathbb{C}} \rightarrow \widehat{\mathbb{C}}$ is a Latt\`{e}s map. Then there exists an imaginary quadratic field $K$ such that the multipliers of $f$ all lie in $\mathcal{O}_{K}$. Furthermore, the multipliers of $f$ are all integers if and only if $f$ is flexible.
\end{proposition}

We are interested in the converse of Proposition~\ref{proposition:specialPC} and Proposition~\ref{proposition:specialL}. In~\cite{M2006}, Milnor conjectured that power maps, Chebyshev maps and flexible Latt\`{e}s maps are the only rational maps that have only integer multipliers. This problem was first studied by the author in~\cite{H2021}, where the conjecture was proved for unicritical polynomial maps and cubic polynomial maps with symmetries. More generally, we may wonder whether power maps, Chebyshev maps and Latt\`{e}s maps are the only rational maps whose multipliers all lie in the ring of integers of a given imaginary quadratic field. In~\cite{H2022}, the author gave a positive answer to this question for quadratic rational maps. The general case was finally settled by Ji and Xie, who thus proved Milnor's conjecture.

\begin{theorem}[{\cite[Theorem~1.13]{JX2022}}]
\label{theorem:milnor}
Assume that $K$ is an imaginary quadratic field and $f \colon \widehat{\mathbb{C}} \rightarrow \widehat{\mathbb{C}}$ is a rational map of degree $d \geq 2$ whose multipliers all lie in $\mathcal{O}_{K}$. Then $f$ is a power map, a Chebyshev map or a Latt\`{e}s map.
\end{theorem}

In fact, the author proved in~\cite{H2021} that every unicritical polynomial map that has only rational multipliers is either a power map or a Chebyshev map. Thus, we may wonder whether Theorem~\ref{theorem:milnor} still holds if $K$ is an arbitrary number field and the multipliers of $f$ are only assumed to lie in $K$. We provide a positive answer to this question.

\begin{theorem}
\label{theorem:main}
Assume that $K$ is a number field and $f \colon \widehat{\mathbb{C}} \rightarrow \widehat{\mathbb{C}}$ is a rational map of degree $d \geq 2$ whose multipliers all lie in $K$. Then $f$ is a power map, a Chebyshev map or a Latt\`{e}s map.
\end{theorem}

We note that Ji and Xie's proof of Theorem~\ref{theorem:milnor} makes crucial use of the fact that the multipliers all lie in a discrete subring of $\mathbb{C}$, whereas the set of multipliers may a priori have limit points under the assumption of Theorem~\ref{theorem:main}. Therefore, we take a different approach. Our proof of Theorem~\ref{theorem:main} relies on an equidistribution result for points of small height proved by Autissier in~\cite{A2001} and on a characterization of power maps, Chebyshev maps and Latt\`{e}s maps proved by Zdunik in~\cite{Z2014}.

Finally, let us mention that Eremenko and van~Strien investigated the rational maps that have only real multipliers in~\cite{EvS2011}. They proved that a rational map $f \colon \widehat{\mathbb{C}} \rightarrow \widehat{\mathbb{C}}$ of degree $d \geq 2$ has this property if and only if it is a flexible Latt\`{e}s map or its Julia set $\mathcal{J}_{f}$ is contained in a circle.

\begin{acknowledgments}
The author would like to thank Xavier Buff, Thomas Gauthier, Igors Gorbovickis and Jasmin Raissy for interesting discussions about this work.
\end{acknowledgments}

\section{Proof of the result}

First, let us state the equidistribution result and the characterization of power maps, Chebyshev maps and Latt\`{e}s maps that we use in our proof of Theorem~\ref{theorem:main}.

We denote by $\overline{\mathbb{Q}}$ the field of algebraic numbers. If $f \colon \widehat{\mathbb{C}} \rightarrow \widehat{\mathbb{C}}$ is a rational map of degree $d \geq 2$ defined over $\overline{\mathbb{Q}}$, then its periodic points all lie in $\overline{\mathbb{Q}} \cup \lbrace \infty \rbrace$. Given a number field $K$, we extend each $\sigma \in \Gal\left( \overline{\mathbb{Q}}/K \right)$ to a map $\sigma \colon \overline{\mathbb{Q}} \cup \lbrace \infty \rbrace \rightarrow \overline{\mathbb{Q}} \cup \lbrace \infty \rbrace$ by setting $\sigma(\infty) = \infty$.

Given a rational map $f \colon \widehat{\mathbb{C}} \rightarrow \widehat{\mathbb{C}}$ of degree $d \geq 2$, we denote by $\mu_{f}$ its measure of maximal entropy, which is a Borel probability measure on $\widehat{\mathbb{C}}$ whose support is the Julia set $\mathcal{J}_{f}$ of $f$. We refer the reader to~\cite{FLM1983}, \cite{L1983} and~\cite{M1983} for further information.

Recall that a sequence $\left( \mu_{n} \right)_{n \geq 0}$ of Borel probability measures on $\widehat{\mathbb{C}}$ \emph{converges weakly} to a Borel probability measure $\nu$ on $\widehat{\mathbb{C}}$ if \[ \lim_{n \rightarrow +\infty} \int_{\widehat{\mathbb{C}}} \varphi \, d\mu_{n} = \int_{\widehat{\mathbb{C}}} \varphi \, d\nu \] for all continuous functions $\varphi \colon \widehat{\mathbb{C}} \rightarrow \mathbb{R}$.

The statement below is a particular case of an equidistribution result for points of small height due to Autissier. We refer to~\cite[Theorem~10.24]{BR2010} and~\cite[Th\'{e}or\`{e}me~2 and Th\'{e}or\`{e}me~4]{FRL2006} for an analogous result in a dynamical context that explicitly implies the statement below.

\begin{theorem}[{\cite[Proposition~4.1.4]{A2001}}]
\label{theorem:equidistribution}
Assume that $f \colon \widehat{\mathbb{C}} \rightarrow \widehat{\mathbb{C}}$ is a rational map of degree $d \geq 2$ defined over a number field $K$ and $\left( S_{n} \right)_{n \geq 0}$ is a sequence of nonempty, pairwise disjoint and $\Gal\left( \overline{\mathbb{Q}}/K \right)$\nobreakdash-invariant sets of periodic points for $f$. For $n \geq 0$, define the Borel probability measure \[ \mu_{n} = \frac{1}{\left\lvert S_{n} \right\rvert} \sum_{z \in S_{n}} \delta_{z} \, \text{.} \] Then $\left( \mu_{n} \right)_{n \geq 0}$ converges weakly to $\mu_{f}$.
\end{theorem}

Given a rational map $f \colon \widehat{\mathbb{C}} \rightarrow \widehat{\mathbb{C}}$ and a periodic point $z_{0} \in \widehat{\mathbb{C}}$ for $f$ with period $p$ and multiplier $\lambda$, the \emph{characteristic exponent} of $f$ at $z_{0}$ is \[ \chi_{f}\left( z_{0} \right) = \frac{1}{p} \log\lvert \lambda \rvert \, \text{.} \]

Suppose that $f \colon \widehat{\mathbb{C}} \rightarrow \widehat{\mathbb{C}}$ is a rational map of degree $d \geq 2$. We denote by $\left\lVert f^{\prime} \right\rVert$ the norm of the differential of $f$ with respect to the spherical metric $\frac{2 \lvert dz \rvert}{1 +\lvert z \rvert^{2}}$; this is the unique continuous function $\left\lVert f^{\prime} \right\rVert \colon \widehat{\mathbb{C}} \rightarrow \mathbb{R}_{\geq 0}$ that satisfies \[ \left\lVert f^{\prime}(z) \right\rVert = \frac{\left\lvert f^{\prime}(z) \right\rvert \left( 1 +\lvert z \rvert^{2} \right)}{1 +\left\lvert f(z) \right\rvert^{2}} \] for all $z \in \mathbb{C}$ such that $f(z) \in \mathbb{C}$. The \emph{Lyapunov exponent} of $f$ is \[ \mathcal{L}_{f} = \int_{\widehat{\mathbb{C}}} \log\left\lVert f^{\prime} \right\rVert \, d\mu_{f} \, \text{,} \] where $\mu_{f}$ denotes the measure of maximal entropy of $f$.

In our proof of Theorem~\ref{theorem:main}, we use the following characterization of power maps, Chebyshev maps and Latt\`{e}s maps:

\begin{theorem}[{\cite[Proposition~4]{Z2014}}]
\label{theorem:exceptional}
Assume that $f \colon \widehat{\mathbb{C}} \rightarrow \widehat{\mathbb{C}}$ is a rational map of degree $d \geq 2$ that is not a power map, a Chebyshev map or a Latt\`{e}s map. Then $f$ has infinitely many periodic points with characteristic exponent greater than $\mathcal{L}_{f}$.
\end{theorem}

\begin{remark}
In fact, Zdunik only states that every rational map $f \colon \widehat{\mathbb{C}} \rightarrow \widehat{\mathbb{C}}$ of degree $d \geq 2$ that is not a power map, a Chebyshev map or a Latt\`{e}s map has at least one periodic point with characteristic exponent greater than $\mathcal{L}_{f}$, but her proof can be easily modified to obtain the statement above. As Theorem~\ref{theorem:exceptional} plays a key role in our proof of Theorem~\ref{theorem:main}, let us explain how Zdunik's proof can be changed. Fix an integer $N \geq 1$. On~\cite[p.260]{Z2014}, denote by $p_{1}, \dotsc, p_{s}$ not only the critical values for $f^{M}$ but also the periodic points for $f$ with period at most $N$ in order to obtain a family $\mathcal{B}$ of balls that also contain no periodic point for $f$ with period less than or equal to $N$. Then the remainder of~\cite[Proof of Proposition~4]{Z2014} shows that there exist a ball $B \in \mathcal{B}$ and a periodic point $z_{0} \in B$ for $f$ such that $\chi_{f}\left( z_{0} \right) > \mathcal{L}_{f}$. The point $z_{0}$ has period greater than $N$ by the construction of $\mathcal{B}$. Since this holds for every $N \geq 1$, the map $f$ has infinitely many periodic points with characteristic exponent greater than $\mathcal{L}_{f}$.
\end{remark}

Given a rational map $f \colon \widehat{\mathbb{C}} \rightarrow \widehat{\mathbb{C}}$, the \emph{postcritical set} of $f$ is \[ \mathcal{P}_{f} = \bigcup_{n \geq 1} f^{\circ n}\left( \mathcal{C}_{f} \right) \, \text{,} \] where $\mathcal{C}_{f}$ denotes the set of critical points for $f$.

We also need the following approximation lemma:

\begin{lemma}
\label{lemma:homoclinic}
Suppose that $f \colon \widehat{\mathbb{C}} \rightarrow \widehat{\mathbb{C}}$ is a rational map of degree $d \geq 2$ and $z_{0} \in \widehat{\mathbb{C}}$ is a repelling periodic point for $f$ that does not lie in $\mathcal{P}_{f}$. Then there is a sequence $\left( w_{n} \right)_{n \geq 0}$ of periodic points for $f$ with pairwise distinct periods such that \[ \lim_{n \rightarrow +\infty} \chi_{f}\left( w_{n} \right) = \chi_{f}\left( z_{0} \right) \, \text{.} \]
\end{lemma}

\begin{proof}
Replacing $f$ by an iterate if necessary, we may assume that $z_{0}$ is a repelling fixed point for $f$. Then there exist connected open neighborhoods $U$ and $V$ of $z_{0}$ such that $\overline{U} \subset V$ and $f$ has no critical point in $\overline{U}$ and induces a biholomorphism $f_{U} \colon U \rightarrow V$. Moreover, $z_{0}$ lies in the Julia set $\mathcal{J}_{f}$ of $f$, its iterated preimages form a dense subset of $\mathcal{J}_{f}$ and $\mathcal{J}_{f}$ has no isolated point, and hence there exist $l \geq 1$ and $z_{l} \in V \setminus \left\lbrace z_{0} \right\rbrace$ such that $f^{\circ l}\left( z_{l} \right) = z_{0}$. The point $z_{l}$ is not critical for $f^{\circ l}$ because $z_{0}$ does not lie in $\mathcal{P}_{f}$ by hypothesis, and hence there is an open neighborhood $W$ of $z_{l}$ such that $\overline{W} \subset V$ and $f^{\circ l}$ has no critical point in $\overline{W}$ and induces a biholomorphism $\left( f^{\circ l} \right)_{W} \colon W \rightarrow f^{\circ l}(W)$. Denote by $g \colon V \rightarrow U$ and $h \colon f^{\circ l}(W) \rightarrow W$ the inverses of $f_{U}$ and $\left( f^{\circ l} \right)_{W}$, respectively. The map $g$ induces a contraction of $\overline{U}$ with respect to the Poincar\'{e} metric on $V$ as $\overline{U} \subset V$. Therefore, $z_{0}$ is the unique periodic point for $g$ and there exists an integer $N \geq l$ such that $g^{\circ (n -l)}(V) \subset f^{\circ l}(W)$ for all $n \geq N$ since $g\left( z_{0} \right) = z_{0}$ and $f^{\circ l}(W)$ is a neighborhood of $z_{0}$. For $n \geq N$, we can consider $h \circ g^{\circ (n -l)} \colon V \rightarrow W$. For every $n \geq N$, the map $h \circ g^{\circ (n -l)}$ induces a contraction of $\overline{W}$ with respect to the Poincar\'{e} metric on $V$ since $\overline{W} \subset V$, and hence it has a unique fixed point $w_{n} \in W$. Let us prove that $w_{n}$ is periodic for $f$ with period $n$ for all $n \geq \max\lbrace N, 2 l \rbrace$. For every $n \geq N$ and every $j \in \lbrace l, \dotsc, n \rbrace$, we have \[ f^{\circ j}\left( w_{n} \right) = f^{\circ j} \circ h \circ g^{\circ (n -l)}\left( w_{n} \right) = g^{\circ (n -j)}\left( w_{n} \right) \, \text{.} \] It follows that $w_{n}$ is periodic for $f$ with period dividing $n$ for all $n \geq N$. Suppose now that $n \geq \max\lbrace N, 2 l \rbrace$ and the period of $w_{n}$ is a proper divisor $p_{n}$ of $n$. Then $n -p_{n} \in \lbrace l, \dotsc, n -1 \rbrace$ and $f^{\circ \left( n -p_{n} \right)}\left( w_{n} \right) = w_{n}$, and hence $w_{n}$ is periodic for $g$ by the relation above. Therefore, $w_{n} = z_{0}$ since $z_{0}$ is the unique periodic point for $g$, and hence \[ z_{0} = h \circ g^{\circ (n -l)}\left( z_{0} \right) = h\left( z_{0} \right) = z_{l} \] by the definition of $w_{n}$, which is a contradiction. Thus, $w_{n}$ is periodic for $f$ with period $n$ for all $n \geq \max\lbrace N, 2 l \rbrace$. Finally, it remains to prove that \[ \lim_{n \rightarrow +\infty} \chi_{f}\left( w_{n} \right) = \chi_{f}\left( z_{0} \right) \, \text{.} \] For $n \geq \max\lbrace N, 2 l \rbrace$, denote by $\lambda_{n}$ the multiplier of $f$ at $w_{n}$, so that \[ \left\lvert \lambda_{n} \right\rvert = \left\lVert \left( f^{\circ l} \right)^{\prime}\left( w_{n} \right) \right\rVert \cdot \prod_{j = l}^{n -1} \left\lVert f^{\prime}\left( f^{\circ j}\left( w_{n} \right) \right) \right\rVert = \left\lVert \left( f^{\circ l} \right)^{\prime}\left( w_{n} \right) \right\rVert \cdot \prod_{j = 1}^{n -l} \left\lVert f^{\prime}\left( g^{\circ j}\left( w_{n} \right) \right) \right\rVert \, \text{.} \] Suppose that $\alpha \in (1, +\infty)$. Since $\left\lVert f^{\prime} \right\rVert$ is continuous at $z_{0}$, there is a neighborhood $O_{\alpha}$ of $z_{0}$ such that \[ \frac{\exp\left( \chi_{f}\left( z_{0} \right) \right)}{\alpha} \leq \left\lVert f^{\prime}(z) \right\rVert \leq \alpha \exp\left( \chi_{f}\left( z_{0} \right) \right) \] for all $z \in O_{\alpha}$. As $g\left( z_{0} \right) = z_{0}$ and $g$ induces a contraction of $\overline{U}$ with respect to the Poincar\'{e} metric on $V$, there exists $J_{\alpha} \geq 1$ such that $g^{\circ j}(V) \subset O_{\alpha}$ for all $j \geq J_{\alpha}$. Define \[ m = \min\left\lbrace \min_{\overline{U}} \left\lVert f^{\prime} \right\rVert, \min_{\overline{W}} \left\lVert \left( f^{\circ l} \right)^{\prime} \right\rVert \right\rbrace \in \mathbb{R}_{> 0} \quad \text{and} \quad M = \max_{\widehat{\mathbb{C}}} \left\lVert f^{\prime} \right\rVert \in \mathbb{R}_{> 0} \, \text{.} \] For every $n \geq \max\left\lbrace N, 2 l, J_{\alpha} +l -1 \right\rbrace$, we have \[ m^{J_{\alpha}} \left( \frac{\exp\left( \chi_{f}\left( z_{0} \right) \right)}{\alpha} \right)^{n -J_{\alpha} -l +1} \leq \left\lvert \lambda_{n} \right\rvert \leq M^{J_{\alpha} +l -1} \left( \alpha \exp\left( \chi_{f}\left( z_{0} \right) \right) \right)^{n -J_{\alpha} -l +1} \, \text{,} \] and hence \[ \chi_{f}\left( w_{n} \right) \geq \frac{J_{\alpha} \log(m)}{n} +\frac{\left( n -J_{\alpha} -l +1 \right) \left( \chi_{f}\left( z_{0} \right) -\log(\alpha) \right)}{n} \] and \[ \chi_{f}\left( w_{n} \right) \leq \frac{\left( J_{\alpha} +l -1 \right) \log(M)}{n} +\frac{\left( n -J_{\alpha} -l +1 \right) \left( \chi_{f}\left( z_{0} \right) +\log(\alpha) \right)}{n} \, \text{.} \] Therefore, we have \[ \chi_{f}\left( z_{0} \right) -\log(\alpha) \leq \liminf_{n \rightarrow +\infty} \chi_{f}\left( w_{n} \right) \leq \limsup_{n \rightarrow +\infty} \chi_{f}\left( w_{n} \right) \leq \chi_{f}\left( z_{0} \right) +\log(\alpha) \, \text{.} \] As this holds for all $\alpha \in (1, +\infty)$, the lemma is proved.
\end{proof}

\begin{remark}
On the one hand, the statement of Lemma~\ref{lemma:homoclinic} is rather weak. Using the Koenigs linearization theorem and the same technique as in the proof above, Ji and Xie show in~\cite[Section~2]{JX2022} that, if $f \colon \widehat{\mathbb{C}} \rightarrow \widehat{\mathbb{C}}$ is a rational map of degree $d \geq 2$ and $z_{0} \in \widehat{\mathbb{C}} \setminus \mathcal{P}_{f}$ is a repelling fixed point for $f$ with multiplier $\lambda$, then there is a sequence $\left( w_{n} \right)_{n \geq N}$ such that $w_{n}$ is periodic for $f$ with period $n$ and multiplier $\rho_{n}$ for all $n \geq N$ and $\rho_{n} = a \lambda^{n} +b +o(1)$ as $n \rightarrow +\infty$ for some $a \in \mathbb{C}^{*}$ and $b \in \mathbb{C}$.

On the other hand, our proof of Lemma~\ref{lemma:homoclinic} can be generalized by using several periodic points. Thus, one can show that, if $f \colon \widehat{\mathbb{C}} \rightarrow \widehat{\mathbb{C}}$ is a rational map of degree $d \geq 2$, then the closure of the set of characteristic exponents of $f$ at its cycles is the union of an interval and a finite set (compare~\cite[Proof of Lemma~2.1]{FG2022}). This is a weak version of~\cite[Corollary~1.16]{JX2022}, which also states that this interval is not a singleton if the map is not a power map, a Chebyshev map or a Latt\`{e}s map.
\end{remark}

We now prove our result.

\begin{proof}[Proof of Theorem~\ref{theorem:main}]
By~\cite[Theorem~2.1]{S1998}, the moduli space $\mathcal{M}_{d}(\mathbb{C})$ of rational maps of degree $d$ is an algebraic variety defined over $\mathbb{Q}$. Denote by $\mathcal{Z}_{f}$ the set of conjugacy classes $[g] \in \mathcal{M}_{d}(\mathbb{C})$ such that $f$ and $g$ have the same multipliers with the same multiplicities at their cycles with period $n$ for each $n \geq 1$. Then $\mathcal{Z}_{f}$ is a Zariski closed subset of $\mathcal{M}_{d}(\mathbb{C})$ defined over $K$ by~\cite[Theorem~4.5]{S1998}. Moreover, it consists of finitely many elements of $\mathcal{M}_{d}(\mathbb{C})$ and possibly a curve of conjugacy classes of flexible Latt\`{e}s maps by~\cite[Corollary~2.3]{MM1987}. Therefore, each rational map $g \colon \widehat{\mathbb{C}} \rightarrow \widehat{\mathbb{C}}$ of degree $d$ such that $[g] \in \mathcal{Z}_{f}$ is conjugate to a rational map defined over a finite extension of $K$ or is a flexible Latt\`{e}s map, and this holds in particular for $f$. If $f$ is a flexible Latt\`{e}s map, we are done. Thus, conjugating $f$ and replacing $K$ by a finite extension if necessary, we may assume that $f$ is defined over $K$.

Suppose that $z_{0} \in \widehat{\mathbb{C}}$ is a repelling periodic point for $f$ that does not lie in $\mathcal{P}_{f}$, and let us prove that $\chi_{f}\left( z_{0} \right) \leq \mathcal{L}_{f}$. By Lemma~\ref{lemma:homoclinic}, there is a sequence $\left( w_{n} \right)_{n \geq 0}$ of periodic points for $f$ with pairwise distinct periods such that \[ \lim_{n \rightarrow +\infty} \chi_{f}\left( w_{n} \right) = \chi_{f}\left( z_{0} \right) \, \text{.} \] For $n \geq 0$, denote by $p_{n}$ and $\lambda_{n}$ the period and multiplier of $w_{n}$ for $f$, respectively, and define \[ S_{n} = \left\lbrace \sigma\left( f^{\circ j}\left( w_{n} \right) \right) : j \in \left\lbrace 0, \dotsc, p_{n} -1 \right\rbrace, \, \sigma \in \Gal\left( \overline{\mathbb{Q}}/K \right) \right\rbrace \] to be the smallest $\Gal\left( \overline{\mathbb{Q}}/K \right)$\nobreakdash-invariant subset of $\overline{\mathbb{Q}} \cup \lbrace \infty \rbrace$ that contains the cycle for $f$ containing $w_{n}$. As $f$ is defined over $K$, \[ S_{n} = \bigcup_{\sigma \in \Gal\left( \overline{\mathbb{Q}}/K \right)} \left\lbrace f^{\circ j}\left( \sigma\left( w_{n} \right) \right) : j \in \left\lbrace 0, \dotsc, p_{n} -1 \right\rbrace \right\rbrace \] is a union of cycles for $f$ with period $p_{n}$ for all $n \geq 0$. In particular, $S_{m} \cap S_{n} = \varnothing$ for all distinct $m, n \geq 0$ since $p_{m} \neq p_{n}$. For every $n \geq 0$ and every $\sigma \in \Gal\left( \overline{\mathbb{Q}}/K \right)$, denoting by $\rho_{\sigma}^{(n)}$ the multiplier of $f$ at $\sigma\left( w_{n} \right)$, we have $\rho_{\sigma}^{(n)} = \sigma\left( \lambda_{n} \right)$ as $f$ is defined over $K$ and $\sigma\left( w_{n} \right)$ has period $p_{n}$, and hence $\rho_{\sigma}^{(n)} = \lambda_{n}$ since $\lambda_{n} \in K$ by hypothesis. Therefore, for every $n \geq 0$, we have \[ \frac{1}{\left\lvert S_{n} \right\rvert} \sum_{z \in S_{n}} \log\left\lVert f^{\prime}(z) \right\rVert = \frac{1}{\left\lvert S_{n} \right\rvert} \sum_{l = 1}^{r} \log\left\lvert \rho_{\sigma_{l}}^{(n)} \right\rvert = \frac{r}{\left\lvert S_{n} \right\rvert} \log\left\lvert \lambda_{n} \right\rvert = \chi_{f}\left( w_{n} \right) \, \text{,} \] where $\sigma_{1}, \dotsc, \sigma_{r}$ are elements of $\Gal\left( \overline{\mathbb{Q}}/K \right)$~-- which depend on $n$~-- such that \[ S_{n} = \bigsqcup_{l = 1}^{r} \left\lbrace f^{\circ j}\left( \sigma_{l}\left( w_{n} \right) \right) : j \in \left\lbrace 0, \dotsc, p_{n} -1 \right\rbrace \right\rbrace \, \text{.} \] For $n \geq 0$, define the Borel probability measure \[ \mu_{n} = \frac{1}{\left\lvert S_{n} \right\rvert} \sum_{z \in S_{n}} \delta_{z} \, \text{.} \] For every $m \in \mathbb{R}$ and every $n \geq 0$, we have \[ \chi_{f}\left( w_{n} \right) \leq \frac{1}{\left\lvert S_{n} \right\rvert} \sum_{z \in S_{n}} \max\left\lbrace \log\left\lVert f^{\prime}(z) \right\rVert, m \right\rbrace = \int_{\widehat{\mathbb{C}}} \max\left( \log\left\lVert f^{\prime} \right\rVert, m \right) \, d\mu_{n} \, \text{.} \] By Theorem~\ref{theorem:equidistribution}, the sequence $\left( \mu_{n} \right)_{n \geq 0}$ converges weakly to $\mu_{f}$. Therefore, for every $m \in \mathbb{R}$, we obtain \[ \chi_{f}\left( z_{0} \right) \leq \int_{\widehat{\mathbb{C}}} \max\left( \log\left\lVert f^{\prime} \right\rVert, m \right) \, d\mu_{f} \] by letting $n \rightarrow +\infty$ since $\max\left( \log\left\lVert f^{\prime} \right\rVert, m \right)$ is continuous on $\widehat{\mathbb{C}}$. Taking the limit as $m \rightarrow -\infty$, it follows from the monotone convergence theorem that $\chi_{f}\left( z_{0} \right) \leq \mathcal{L}_{f}$. Thus, we have shown that $\chi_{f}\left( z_{0} \right) \leq \mathcal{L}_{f}$ for all but finitely many periodic points $z_{0} \in \widehat{\mathbb{C}}$ for $f$. Therefore, $f$ is a power map, a Chebyshev map or a Latt\`{e}s map by Theorem~\ref{theorem:exceptional}, which completes the proof of the theorem.
\end{proof}

\providecommand{\bysame}{\leavevmode\hbox to3em{\hrulefill}\thinspace}
\providecommand{\MR}{\relax\ifhmode\unskip\space\fi MR }
\providecommand{\MRhref}[2]{%
	\href{http://www.ams.org/mathscinet-getitem?mr=#1}{#2}
}
\providecommand{\href}[2]{#2}

\end{document}